\documentclass[10pt]{article}

\usepackage{amsmath}
\usepackage{amssymb}
\usepackage{amsthm}
\usepackage{amscd}

\newtheoremstyle{plain2}{\topsep}{\topsep}%
     {\itshape}
     {}
     {\bfseries}
     {.}
     {.5em}
     {\thmnumber{(#2)}\thmname{ #1}\thmnote{ #3}}

\theoremstyle{plain2}
\newtheorem{teo}{Theorem}[section]
\newtheorem{prop}[teo]{Proposition}
\newtheorem{coro}[teo]{Corollary}
\newtheorem{lemma}[teo]{Lemma}

\newtheoremstyle{definition2}{\topsep}{\topsep}%
     {}
     {}
     {\bfseries}
     {.}
     {.5em}
     {\thmnumber{(#2)}\thmname{ #1}\thmnote{ #3}}

\theoremstyle{definition2}

\newtheorem{rem}[teo]{Remark}

\def\R{\mathbb{R}}

\def\Vb{\mathbb{V}}

\def\eps{\varepsilon}
\def\vphi{\varphi}

\def\spann{\mathrm{span}}
\def\proj{\mathrm{proj}}
\def\Neh{\mathcal{N}}

\title{A remark on natural constraints in variational methods and an application to superlinear
Schr\"odinger systems.\footnote{Work
partially supported by the PRIN2009 grant ``Critical Point Theory and
Perturbative Methods for Nonlinear Differential Equations''. \break 
2010 \emph{AMS Subject Classification.} Primary 58E05, 35A15; secondary 35J50.
}}
\author{Benedetta Noris and Gianmaria Verzini}

\date{December 14, 2011.}

\begin{document}

\maketitle

\begin{abstract}
For a $C^2$-functional $J$ defined on a Hilbert space $X$, we consider the set
\(
\Neh=\{ x\in A:\, \proj_{V_x}\nabla J(x)=0\},
\)
where $A\subset X$ is open and $V_x\subset X$ is a closed linear subspace, possibly
depending on $x\in A$. We study sufficient conditions for a constrained critical point
of $J$ restricted to $\Neh$ to be a free critical point of $J$, providing a unified approach to
different natural constraints known in the literature, such as the Birkhoff-Hestenes
natural isoperimetric conditions and the Nehari manifold. As an application,
we prove multiplicity of solutions to a class of superlinear Schr\"odinger systems on singularly
perturbed domains.
\end{abstract}


\section{Introduction}

Let $X$ denote a Hilbert space and let $J$ be a functional of class $C^2$ on $X$.
A \emph{natural constraint} for $J$ is a manifold $\Neh\subset X$ enjoying the property
that every critical point of $J$ constrained to $\Neh$ is in fact a free critical point.
When searching for critical points of $J$, natural constraints are typically used when $J$
does not admit (nontrivial) minima.

According to Birkhoff and Hestenes \cite{birkhoff1}, the first example in the literature
appears in a paper by Poincar\'e \cite{poincare05} in the study of closed geodesics
on a closed convex analytic surface. Since such a geodesic can not be of minimal length,
Poincar\'e finds it by minimizing the length functional among paths which satisfy the
``natural isoperimetric condition'' of dividing the surface into two parts of equal integral curvature.
The aforementioned paper by Birkhoff and Hestenes is the first which considers natural constraints
from an abstract point of view. In particular, for the fixed end-point problem in the calculus of
variations, the authors prove that every extremal of the functional is indeed a local minimum when a suitable
set of natural conditions is added. In modern language, these conditions can be written as
\[
\langle \nabla J(x), \xi_i \rangle = 0,
\]
where the variations $\xi_i$'s are related to the second variation of $J$, and their number is the
Morse index of the extremal.

In more recent times, apart from natural constraints induced by symmetry \cite{palais79}, the development of this topic followed mainly two directions.
On one hand, the ideas of Birkhoff and Hestenes were exploited by Berger \cite{berger71}
in searching for periodic orbits to Hamiltonian systems, and
by Berger and Schechter \cite{berger77} from a more abstract point of view. In these papers
$\nabla J$ has a semilinear structure, while the natural constraint is
\[
\Neh_{\text{B-S}} = \left\{x\in X:\, \langle \nabla J(x),v\rangle = 0\text{ for every }v\in V \right\}
\]
where $V\subset X$ is a closed linear subspace such that $J''(x)$ is definite on $V$ for every
$x\in\Neh_{\text{B-S}}$. This implies both that
$\Neh_{\text{B-S}}$ is a manifold and that it is a natural constraint, because for any constrained critical
point the corresponding Lagrange multiplier is zero.
On the other hand, one of the most famous examples
of natural constraint is the so called \emph{Nehari manifold}
\[
\Neh_{\text{N}} = \left\{x\in X:\, x\neq 0\text{ and }G(x):=\langle \nabla J(x),x\rangle = 0 \right\},
\]
which is named after the papers by Zeev Nehari \cite{nehari60,nehari61,nehari75}.
Again, since
\begin{equation}\label{eq:intro}
\langle \nabla G(x), x\rangle = \langle \nabla J(x), x\rangle + J''(x)[x,x],
\end{equation}
if $J''(x)$ along $x$ is non-degenerate for every $x\in\Neh$, then
both $\Neh$ is a manifold and it is a natural constraint
(see for instance \cite{ambrosetti}, Proposition 1.4).

As we mentioned before, a natural constraint $\Neh$ is particularly useful when searching for
non-minimal critical points, which are minima of the restricted functional, so that one expects
to find critical points of $J$ by minimizing $J|_{\Neh}$.
From this point of view, the two types of natural constraints introduced above behave in a quite different way.
While $\Neh_{\text{B-S}}$ is often weakly closed, so that
the direct method of the calculus of variations usually applies, on the other hand $\Neh_{\text{N}}$ needs not be,
thus exhibiting a lack of compactness.
The typical strategy to overcome this difficulty is to provide a sort of ``projection'' of $X\setminus\{0\}$ into $\Neh_{\text{N}}$, such as $u\mapsto \bar t (u)u$, where $\bar t(u)\in\R$ is conveniently chosen by studying the critical points of the function $t\mapsto J(tu)$.
In this direction, the main problems arise when a globally defined projection is not available.
An alternative way to proceed is to show that $J|_{\Neh_{\text{N}}}$ satisfies the Palais-Smale condition. To this aim
it is sufficient to require that the non-degeneracy of $J''$ holds uniformly on $\Neh_{\text{N}}$ in the sense that
\begin{equation}\label{eq:intro2}
\text{either } J''(x)[x,x]\geq\delta \|x\|^2 \text{ or } J''(x)[x,x]\leq-\delta \|x\|^2,
\end{equation}
for some $\delta>0$, for every $x\in \Neh_{\text{N}}$. Indeed, under such assumption, one can prove that constrained Palais-Smale sequences are
free ones. This allows to recover compactness by assuming the usual Palais-Smale condition on $J$ (see for instance \cite{dancerWeiWeth2010}).

In the literature it is possible to find a number of generalizations of the above ideas: when searching for points such that $\nabla J(x)=0$, one imposes as a preliminary condition the vanishing of the projection of $\nabla J(x)$ on some closed subspace $V_x\subset X$, possibly dependent on $x$. This is the case of the classical Nehari manifold, since
\[
\langle \nabla J(x),x\rangle=0 \quad\iff\quad \proj_{\spann\{x\}} \nabla J(x)=0.
\]
Among others, we wish to mention
\cite{vangroesen88,ramosPistoia04,ramosYang05,pankov07,ramosTavares2008,
giraoGomes09,szulkinWeth2009,szulkinWeth2010}.

The main aim of the present paper is to provide conditions in order to extend the above scheme to the constraint
\[
\Neh=\{ x\in A:\, G(x):=\proj_{V_x}\nabla J(x)=0\},
\]
where $A\subset X$ is open and $V_x\subset X$ is a closed linear subspace, for every $x\in A$. Referring to \eqref{eq:intro} and \eqref{eq:intro2}, the main feature we want to preserve is that the differential of $G$ restricted to $V_x$ consists of two terms, one of which vanishes on $\Neh$, and the other one is a quadratic form related to $J''$, enjoying some coercivity property.
It will come out that, apart from some regularity conditions, we will need two main properties, namely that
\begin{itemize}
 \item $V_x$ is \emph{invariant under differentiation}, in the sense that the differential
 of any regular vector field laying in $V_x$ for every $x$ maps $V_x$ into itself;
 \item $V_x$ splits into two subspaces $V^\pm_x$, with the property that $J''(x)$ is \emph{coercive/
 anticoercive} on $V^\pm_x$ respectively.
\end{itemize}
We stress the fact that, with respect to the previous literature, we do not require $J''$ to be definite on $V_x$; this allows a better localization of the critical points, as we show in our application to nonlinear Schr\"odinger systems.
To express the dependence of $V_x$ on $x$, it is useful to introduce a vector bundle
structure on $V$, the disjoint union of $V_x$. To do that, we denote by $TA$ the (trivial)
tangent bundle of $A$. In the following we are interested only in trivial $C^1$-subbundles of $TA$,
that is bundles $V \to A$, with $V \subset TA$, equipped with a global $C^1$-trivialization
\[
\tau: V\to A\times \mathbb{V},
\]
for some Hilbert space $\mathbb{V}$. With this notation, $V_x$ is the fiber of $V$ at $x$
which is isomorphic to $\Vb$ via $\tau_x:=\tau(x,\cdot)$.
We observe that, by means of the natural immersion which we will systematically omit,
$\tau^{-1}$ can be naturally interpreted as a $C^1$ map
\[
\tau^{-1}: A\times \mathbb{V} \to  A \times X.
\]
With this notation, the regularity assumptions we mentioned above concern $\partial_x \tau^{-1}$,
besides $J'$ and $J''$.
Our main result is the following.
\begin{teo}\label{teo:main_intro}
Let $X$ be a Hilbert space and let $J\in C^2(X,\R)$. For $A\subset X$ open, let
$V^\pm$ be two trivial $C^1$-subbundles of $TA$,
with fibers $\Vb^\pm$ and trivializations $\tau^\pm$ respectively, which we assume
to induce isometries $\tau_x^\pm$ on every fiber.
Suppose that $V_x^+\cap V_x^-=\{0\}$ and that $V:=V^++V^-$ is such that $V_x$ is a
proper subspace of $T_xA$.
Set
\[
\Neh :=\left\{x\in A:\,\proj_{V_x}\nabla J(x)=0 \right\}
\]
and assume that there exists $\delta>0$ such that, for every $x\in\Neh$, it holds
\begin{itemize}
 \item[(inv)] $\xi'(x)[v]\in V_x$ for every $(\cdot,\xi(\cdot))$ $C^1$-section of $V$, $v\in V_x$;
 \item[(coe)]  $\pm J''(x)[v,v] \geq \delta \|v\|_X^2$ for every $v\in V^\pm_x$.
\end{itemize}
Furthermore, assume that $J'(x)\in X^*$, $J''(x):X\times X \to \R$, and $\partial_x
(\tau^{\pm}_x)^{-1}:\Vb^\pm\times X\to X$ are bounded as linear/bilinear maps, uniformly
for $x$ in $\Neh$.

Then $\Neh$ is a natural constraint for $J$, and every constrained Palais-Smale sequence
for $J$ is indeed a free one.
\end{teo}
We stress the fact that the above theorem can be exploited in order to obtain the existence
of critical points for $J$, without the need of defining any global projection of $A$ onto $\Neh$.

To better clarify the assumptions above, one can consider the particular case (which includes most applications) in which $V_x$ is constant, except for a finite dimensional subspace. That is,
let us consider a fixed closed subspace $W\subset X$, and $\xi_i\in C^1(A,X)$, $i=1,\dots,k$, an orthonormal set, with
$W \cap \spann\left\{\xi_1(x),\dots,\xi_k(x)\right\}=\{0\}$ and let us set
\[
V_x = W \oplus \spann\left\{\xi_1(x),\dots,\xi_k(x)\right\}.
\]
In such a situation, $V_x$ induces a $C^1$-subbundle of $TA$ with fiber $\mathbb{V}\cong W\times \R^k$
and trivialization
\[
\tau_x v = (\proj_{W} v, \langle v, \xi_1(x)\rangle, \dots , \langle v, \xi_k(x)\rangle).
\]
As a consequence, $\tau_x$ is trivially an isometry and
\[
\partial_x (\tau_x)^{-1}: ((w,t_1,\dots, t_k),u) \mapsto \sum_{i=1}^k t_i \xi'_i(x)[u]
\]
is uniformly bounded as a bilinear map on $\Vb\times X$ whenever the linear operators
$\xi'_i(x):X\to X$ are. Finally, assumption (inv) can be more explicitly written as
\[
\xi'_i(x)[v]\in V_x \quad\text{ for every }v\in V_x,\ 1\leq i\leq k.
\]

One of the main advantages of the method of natural constraints with respect to other
variational methods, such as
mountain pass or linking theorems, is that it allows to better localize the critical points, thus
providing a deeper qualitative description. This is particularly advantageous when facing multiplicity
issues. To illustrate this point, in the second part of the paper we apply the above result in order to prove multiplicity of solutions to a class of elliptic systems of gradient type with superlinear
nonlinearities, in singularly perturbed domains. Despite the fact that we can deal with
more general situations, in this introduction we describe our results in the case of cubic
nonlinearities in a smooth bounded domain of $\R^N$, with $N=2,3$.
Such type of nonlinearities have been extensively studied in the recent years, due to
their applications both to nonlinear optics and to Bose-Einstein condensation. Let us consider the system
\begin{equation}\label{eq:bec}
-\Delta u_i=\mu_i u_i^{3} + u_i\sum_{j\neq i}\beta_{ij}u_j^2, \qquad u_i>0,\ u_i\in H^1_0(\Omega),
\qquad i=1,\ldots,k,
\end{equation}
where $\mu_i>0$, $\beta_{ij}=\beta_{ji}\in\R$, for every $i,j$. At least in some particular cases, system \eqref{eq:bec} is well known to admit positive solutions with minimal energy, see for instance \cite{ctv2002poincare,ctv2003JFA,dancerWeiWeth2010}.
We aim at extending to \eqref{eq:bec} the results first obtained by Dancer in the case of a single equation,
concerning the effect of the domain shape on the  multiplicity
of solutions, see \cite{dancer88,dancer90}.
While in these papers the tools are mainly topological, a variational approach to the single
equation case has been introduced by Beyon in \cite{byeon2001}. We prove the following.
\begin{teo}\label{teo:intro_appl}
Let $\Omega$ and $\Omega_l$, $l=1,\dots,n$, be bounded regular domains such that
\[
\overline{\Omega}_l\cap \overline{\Omega}_m =\emptyset\text{ for every }
l\neq m,\qquad \Omega \setminus \overline{D} = \bigcup_{l=1}^n \Omega_l,
\]
where $D$ is a bounded regular open set which is sufficiently small in a suitable sense.
If $\beta_{ij}\leq \bar \beta$ for every $i\neq j$, with $\bar\beta>0$ sufficiently small,
then system \eqref{eq:bec} admits at least $(2^n-1)^k$ positive solutions.
\end{teo}
We distinguish the solutions because, using suitable natural constraints, we can prescribe
whether $u_i|_{\Omega_l}$ is either large or small, for every $i=1,\dots,k$, $l=1,\dots,n$.
Note that, in particular, our result holds true in the purely competitive case, i.e. $\beta_{ij}<0$.
The smallness of $D$ will be made precise by suitable assumptions in the following; for instance,
the result holds if $D$ can be decomposed in a finite number of parts, each of which
lies between two hyperplanes sufficiently close.
We wish to mention that related systems in similar domains were considered, from a different point
of view, in \cite{contiFelli09}.

\textbf{Notations.} Given $I\in C^k(X,Y)$, $k\geq 1$, with $X$ and $Y$ Hilbert spaces, and
$x_0,u\in X$, we write
\(
I'(x_0)[u]  \in Y
\)
to denote the (first) differential of $I$ evaluated at $x_0$ along $u$. Analogously,
\(
I''(x_0)[u,v] \in Y
\)
will denote the (bilinear form associated to the) second differential along $(u,v)\in X\times X$.
In case $Y=\R$, a sequence $\{x_n\}_n\subset X$ is a Palais-Smale (PS) sequence for $I$ (at level $c$) if
\[
I(x_n)\to c \quad\text{ and }\quad I'(x_n)\to 0 \text{ in } X^*.
\]
$I$ satisfies the PS-condition (at level $c$) if every PS-sequence admits a converging subsequence.
We say that a subspace $V\subset X$ is proper if $V\neq\left\{0\right\}$ and $V\neq X$. The orthogonal projection of a vector $u\in X$ on $V$ will be denoted by $\proj_V u$.
For $\Omega\subset\R^N$ smooth bounded domain, $\Gamma \subset \partial \Omega$ relatively open
and $p\leq 2^*:=2N/(N-2)$ we denote by $C_S(\Omega,p)$ (resp. $C_S(\Omega,\Gamma,p)$) the Sobolev constant
related to the embedding of $H^1_0(\Omega)$ (resp. $H^1_{0,\Gamma}(\Omega)$) into $L^p(\Omega)$.
Finally, we denote by $C$ any constant we need not to specify.

\section{Generalized Nehari manifolds}\label{sec:abstract}
\subsection{Palais-Smale sequences on natural constraints}

Let $X$, $Y$ be Hilbert spaces, $A\subset X$ open and $G\in C^1(A,Y)$. We
denote by $\Neh$ the zero set of $G$, that is
\[
\Neh:=\{x\in A: \ G(x)=0  \}.
\]
Let us recall a well known condition which ensures that $\Neh$ is a manifold.
\begin{prop}
Let $G\in C^1(A,Y)$. If, for every $x\in \Neh$, $G'(x)$ is surjective and $\ker (G'(x))$
is a proper subspace of $X$, then $\Neh$ is a $C^1$-manifold and the tangent space to $\Neh$
at $x$ is $\ker (G'(x))$.
\end{prop}
\begin{proof}[Sketch of the proof]
Being $\ker (G'(x))$ a closed and proper linear subspace, we have the
nontrivial splitting $X=\ker(G'(x))\oplus \ker (G'(x))^{\perp}$.
Now, since $G'(x) : \ker (G'(x))^{\perp} \to Y$ is bijective, the implicit
function theorem applies and this provides a local parametrization of $\Neh$ around $x$.
\end{proof}

Our first aim is to establish some general conditions under which $\Neh$ is a natural
constraint for a functional $J$ defined on $X$.

\begin{prop}\label{prop:abstract_nehari_natural_constraint}
Let $J\in C^1(X,\R)$, $G\in C^1(A,Y)$ and let $\Neh$ be defined as above. Let us
assume that
\[
\text{for every $x\in \Neh$ there exists a closed and proper linear subspace $V_x\subset X$}
\]
such that
\begin{eqnarray}
\label{eq:ass1}&& J'(x)|_{V_x}\text{ is identically zero;}\\
\label{eq:ass2}&& G'(x)|_{V_x}\text{ is surjective onto }Y.
\end{eqnarray}
Then $\Neh$ is a manifold and a natural constraint for $J$.
\end{prop}
\begin{proof}
Let us first show that $\ker (G'(x))$ is a proper subspace of $X$, so that,
by the previous proposition, $\Neh$ is a manifold.
Clearly $\ker (G'(x))$ can not be the entire space, by \eqref{eq:ass2}.
Let $0\neq v_1\in V_x^{\perp}$ (which exists since $V_x$ is proper). By \eqref{eq:ass2}
there exists $v_2\in V_x$ such that $G'(x)[v_1]=G'(x)[v_2]$, hence
$v_1-v_2\in \ker (G'(x))$.
We turn to the second part of the statement.
Let $x_0\in \Neh$ be a critical point of $J$ constrained to $\Neh$.
Then there exists a Lagrange multiplier $\lambda \in Y^*$ such that
\[
J'(x_0)[x]=\lambda\left[G'(x_0)[x]\right] \qquad \text{ for every } x\in X.
\]
In particular we have
\[
\lambda\left[G'(x_0)[v]\right]=J'(x_0)[v]=0 \qquad \text{ for every } v\in V_{x_0},
\]
and being $G'(x_0)$ surjective on $V_{x_0}$ we deduce that $\lambda\equiv0$, i.e.
$x_0$ is a free critical point of $J$.
\end{proof}
When searching for critical points of $J$, a typical strategy consists in selecting
a candidate critical value via some variational principle, and then to exploit
some compactness, usually in the form of a Palais-Smale condition. Since on natural
constraints free critical points coincide with constrained ones, it is
natural to wonder if a similar equivalence holds for Palais-Smale sequences too.
It comes out that, in our setting, while the first property depends on
the surjectivity of $G'|_V$, the latter one leans on the uniform injectivity of
the same operator.
\begin{prop}\label{prop:abstract_nehari_palais_smale}
Under the assumptions of Proposition \ref{prop:abstract_nehari_natural_constraint},
let us assume moreover that there exist positive constants $\rho$, $\rho'$, such that
\begin{gather}
\|G'(x)[v]\|_Y \geq \rho \|v\|_X \quad \text{ for every }x\in \Neh,\, v\in V_x, \label{eq:unif_inject}\\
\|G'(x)[u]\|_Y \leq \rho' \|u\|_X \quad \text{ for every }x\in \Neh,\, u\in X.\label{eq:unif_bound}
\end{gather}
Then, for every sequence $\{(x_n,\lambda_n)\}\subset \Neh\times Y^*$,
\[
J'(x_n)-\lambda_n\circ G'(x_n)\to 0 \ \text{ in } X^* \qquad\implies\qquad
J'(x_n)\to0 \ \text{ in } X^*.
\]
\end{prop}
\begin{proof}
By definition we have
\[
\sup_{\substack{u\in X\\u\neq0}}\frac{|J'(x_n)[u]-\lambda_n[G'(x_n)[u]]|}{\|u\|_X}
= \|J'(x_n)-\lambda_n\circ G'(x_n)\|_{X^*}\to 0.
\]
Since $J'(x_n)[v]=0$ for every $v\in V_{x_n}$, we deduce that
\[
\sup_{\substack{v\in V_{x_n}\\v\neq0}}\frac{|\lambda_n[G'(x_n)[v]]|}{\|v\|_X} \to0.
\]
Now, recalling that $G'(x_n)$ restricted to $V_{x_n}$ is surjective, we deduce that
\[
\|\lambda_n\|_{Y^*}=\sup_{\substack{y\in Y\\y\neq0}}\frac{|\lambda_n[y]|}{\|y\|_Y}=
\sup_{\substack{v\in V_{x_n}\\v\neq0}} \frac{|\lambda_n[G'(x_n)[v]]|}{\|G'(x_n)[v]\|_Y} \leq
\frac{1}{\rho}\sup_{\substack{v\in V_{x_n}\\v\neq0}} \frac{|\lambda_n[G'(x_n)[v]]|}{\|v\|_X}\to0.
\]
Finally, the uniform continuity implies
\[
\sup_{\substack{u\in X\\u\neq0}}\frac{|\lambda_n[G'(x_n)[u]]|}{\|u\|_X} \leq
\|\lambda_n\|_{Y^*} \sup_{\substack{u\in X\\u\neq0}}\frac{\|G'(x_n)[u]\|_Y}{\|u\|_X} \leq
\rho' \|\lambda_n\|_{Y^*},
\]
which concludes the proof.
\end{proof}

Under standard additional assumptions, the previous result ensures the existence of a critical point of $J$ belonging to $\Neh$.

\begin{coro}\label{coro:abst_neh}
In the assumptions of Propositions \ref{prop:abstract_nehari_natural_constraint} and
\ref{prop:abstract_nehari_palais_smale}, suppose moreover that
\[
\inf_{x\in \overline{\Neh}\setminus \Neh} J(x) > \inf_{x\in \Neh} J(x) =: c \in \R
\]
and that $J$ satisfies the Palais-Smale condition at level $c$.
Then there exists $x_0 \in \Neh$ such that $J(x_0)=c$ and $J'(x_0)=0$.
\end{coro}
\begin{proof}
By Ekeland's variational principle \cite{struwe08book} applied to $\overline{\Neh}$ there exists
$\{x_n\}\subset \Neh$ and $\{\lambda_n\}\subset Y^*$ such that
\[
J(x_n)\to c \qquad \text{ and }\qquad J'(x_n)-\lambda_n\circ G'(x_n)\to 0 \ \text{ in } X^*
\]
as $n\to+\infty$. By the previous proposition $J'(x_n)\to0$ in $X^*$, and the conclusion follows in a standard way.
\end{proof}

\begin{rem}\label{rem:only_min_sequences}
In order to prove the previous result it suffices to assume conditions
\eqref{eq:unif_inject} and \eqref{eq:unif_bound} only on minimizing sequences.
\end{rem}

\subsection{Proof of Theorem \ref{teo:main_intro}}

A remarkable particular case of the structure just introduced is when the closed linear
subspaces $V_x$ depend in a smooth way on $x$ and $G(x)$ is the projection of $\nabla J(x)$ on $V_x$.
In this case, assumption \eqref{eq:ass1} on $J'$ is tautologically satisfied, while we will show that assumption \eqref{eq:unif_inject} on $G'$ can be expressed in terms of $J''$, in case the subspaces are invariant under differentiation.

In view of the application of Proposition \ref{prop:abstract_nehari_palais_smale}, we set
\[
\begin{array}{ll}
Y:=\mathbb{V}^+\times\mathbb{V}^-,
&\langle y,z\rangle_Y := \langle y^+,z^+\rangle_{\mathbb{V}^+} + \langle y^-,z^-\rangle_{\mathbb{V}^-},\smallskip\\
G:A\to Y,&G(x):=(G^+(x),-G^-(x)),
\end{array}
\]
where $y=(y^+,y^-)$, $z=(z^+,z^-)$ and
\[
G^\pm(x):=\tau^{\pm}_x\proj_{V^\pm_x}\nabla J(x),
\]
so that the set $\Neh$ which appears in the statement of Theorem \ref{teo:main_intro} is indeed
the null set of $G$.
Notice first that since $V^\pm$ are $C^1$-subbundles of $TA$,
then $G\in C^1(A,Y)$. As a consequence we can evaluate $G'(x)$, first  along directions in $V_x$ and next along
directions in $X$.
\begin{lemma}\label{lemma:main_lemma}
For every $\bar x\in\Neh$, $\bar v^+\in V^+_{\bar x}$, $\bar w\in V_{\bar x}$ it holds
\begin{equation*}
\left\langle (G^+)'(\bar x)[\bar w],\tau^+_{\bar x} \bar v^+\right\rangle_{\mathbb{V}^+} =
J''(\bar x)[\bar w,\bar v^+]
\end{equation*}
(and an analogous property holds for $G^-$).
\end{lemma}
\begin{proof}
Let $\xi(x)=(\tau_x^+)^{-1}\tau_{\bar x}^+\bar v^+$. Note that $(\cdot,\xi(\cdot))$ is $C^1$-section of $V^+$, i.e. $\xi(x)\in V^+_x$ for every $x\in A$, and $\xi(\bar x)=\bar v^+$. By definition of projection we have that
\[
\begin{split}
\langle G^+(x),\tau^+_{\bar x} \bar v^+ \rangle_{\mathbb{V}^+}  &=
\langle \proj_{V^+_x}\nabla J(x),\xi(x) \rangle_X \\
&=\langle \nabla J(x),\xi(x)\rangle_X.
\end{split}
\]
By differentiating the previous expression at $\bar x$, along $\bar w\in V_{\bar x}$, we obtain
\begin{equation}\label{eq:diff_bundle}
\left\langle (G^+)'(\bar x)[\bar w],\tau^+_{\bar x} \bar v^+\right\rangle_{\mathbb{V}^+} =
J''(\bar x)[\bar w,\bar v^+]+\langle \nabla J(\bar x), \xi'(\bar x)[\bar w] \rangle_X,
\end{equation}
where the last term vanishes because of assumption (inv).
\end{proof}
\begin{lemma}\label{lemma:uniform_bound}
There exists a positive constant $\rho'$ such that for every $x\in\Neh$ and $u\in X$ it holds
\begin{equation*}
\|G'(x)[u]\|_Y\leq \rho'\|u\|_X.
\end{equation*}
\end{lemma}
\begin{proof}
To start with, we claim that
there exists a positive constant $\rho''$ such that for every $x\in\Neh$, $v^+\in V^+_{x}$ and $u\in X$ it holds
\begin{equation*}
\left| \left\langle (G^+)'(x)[u],\tau^+_{x} v^+\right\rangle_{\mathbb{V}^+} \right| \leq \rho'' \|v^+\|_X \|u\|_X,
\end{equation*}
and an analogous property holds for $G^-$. Indeed,
reasoning as in the previous lemma, we have that \eqref{eq:diff_bundle} holds with $u\in X$ instead of $\bar w\in V_{\bar x}$. The claim follows, recalling the definition of $\xi$, by the assumptions of uniform boundedness on $J'$, $J''$ and $\partial_x(\tau_x^+)^{-1}$.  Now, by isomorphism, vectors $g^\pm \in V_x^\pm$ are uniquely determined so that
$G'(x)[u]=\tau^+_xg^++\tau_x^-g^-$. With this notation we have
\[
\begin{split}
\|G'(x)[u]\|_Y^2 &=
\langle (G^+)'(x)[u], \tau^+_xg^+ \rangle_Y - \langle (G^+)'(x)[u], \tau_x^-g^- \rangle_Y \\
&\leq \rho'' \|u\|_X (\|g^+\|_X+\|g^-\|_X)= \rho'' \|u\|_X (\|\tau_x^+g^+\|_{\Vb^+}+\|\tau_x^-g^-\|_{\Vb^-})\\
&\leq \sqrt{2} \rho'' \|u\|_X\cdot \|\tau^+_xg^++\tau_x^-g^-\|_Y
=\rho' \|u\|_X\cdot\|G'(x)[u]\|_Y.\qedhere
\end{split}
\]
\end{proof}
We notice that $\tau^\pm$ induce a global $C^1$-trivialization $\tau:V\to A\times Y$, with fiber
\[
\tau_x:V^++V^-\to Y, \qquad
\tau_x:v^++v^-\mapsto (\tau_x^+ v^+,\tau_x^- v^-).
\]
Even though $\tau_x$ needs not to be an isometry, we have that $\langle \tau^+_x v^+,\tau^-_x v^-\rangle_Y=0$ for every $v^\pm\in V_x^\pm$, and hence
\begin{equation*}
\|\tau_x v\|_Y^2=\|v^+\|_X^2+\|v^-\|_X^2\geq\frac{1}{2}\|v\|_X^2.
\end{equation*}
\begin{proof}[Proof of Theorem \ref{teo:main_intro}]
We apply Proposition \ref{prop:abstract_nehari_palais_smale} to our context. Assumption
\eqref{eq:ass1} holds by definition, since
$J'(x)$ identically vanishes along vectors of $V_x$. As it regards \eqref{eq:ass2}, for fixed $x\in\Neh$,
$y\in Y$, we search for $w\in V_x$ such that $G'(x)[w]=y$. This is equivalent
to solving the abstract variational problem
\[
a(w,v) = \langle y , \tau_x v \rangle_Y \qquad\text{for every }v\in V_x,
\]
where $a(w,v)$ is the following bilinear form on $V_x$
\[
\begin{split}
a(w,v) & := \langle G'(x)[w],\tau_{x} v \rangle_{Y} \\
  &= \langle (G^+)'(x)[w],\tau^+_{x} v^+ \rangle_{\mathbb{V}^+} -
     \langle (G^-)'(x)[w],\tau^-_{x} v^- \rangle_{\mathbb{V}^-} \\
  &= J''(x)[w,v^+] - J''(x)[w,v^-]
\end{split}
\]
(in the last equality we used Lemma \ref{lemma:main_lemma}).
Such a problem can be easily solved by applying Lax-Milgram Theorem, since $a(w,v)$ is bounded by Lemma \ref{lemma:uniform_bound} and it is coercive because
\[
\begin{split}
a(v,v) & = J''(x)[v^+,v^+] - J''(x)[v^-,v^-] \\
& \geq \delta (\|v^+\|^2_X + \|v^-\|^2_X)=\delta \|\tau_x v\|_Y^2\geq\frac{\delta}{2}\|v\|^2_X,
\end{split}
\]
where we used the fact that $J''(x)$ is symmetric and assumption (coe). The last calculation also provides the validity of \eqref{eq:unif_inject} as follows
\[
\|G'(x)[v]\|_Y\cdot \|\tau_x v\|_Y \geq a(v,v) \geq \delta \|\tau_x v\|^2_Y \geq \frac{\delta}{\sqrt{2}}\|v\|_X \cdot \|\tau_x v\|_Y.
\]
Finally, \eqref{eq:unif_bound} was proved in Lemma \ref{lemma:uniform_bound}, so that all the assumptions of Proposition \ref{prop:abstract_nehari_palais_smale} hold true.
\end{proof}
%
%
To conclude the section we provide a version of Theorem \ref{teo:main_intro} specialized to
the applications we will present next.
%
\begin{teo}\label{teo:applications_specialized}
Let $X$ be a Hilbert space, $J\in C^2(X,\R)$, $V^+\subset X$ a fixed closed
linear subspace. We define
\[
V^+_x\equiv V^+,\qquad V^-_x:=\spann\left\{\xi_1(x),\dots,\xi_h(x)\right\},\qquad V_x := V^+_x \oplus V^-_x,
\]
with $\xi_i \in C^1(A,X)$ for every $i=1,\dots,h$, $A\subset X$ open, in such a way that $V_x$ is proper.
As usual, let
\[
\Neh :=\left\{x\in A:\,\proj_{V_x}\nabla J(x)=0 \right\}\qquad\text{and}\qquad c:= \inf_{\Neh} J.
\]
Let us suppose that
\begin{enumerate}
 \item[(i)] $c\in \R$, $\inf_{\overline{\Neh}\setminus\Neh} J > c$;
 \item[(ii)] $J$ satisfies the PS-condition at level $c$.
\end{enumerate}
Moreover, let us assume that
for some $0<\delta<\delta'$ there holds, for every $x\in\Neh$ with $J(x)\leq c+1$,
\begin{enumerate}
 \item[(iii)] $\|\xi_i(x)\|_X\geq\delta$, $\langle \xi_i(x),\xi_j(x)\rangle_X=0$,
 for every $i\neq j$;
 \item[(iv)] $\xi'_i(x)[v]\in V_x$ for every $i$ and $v\in V_x$;
 \item[(v)] $\pm J''(x)[v,v] \geq \delta \|v\|_X^2$ for every $v\in V^\pm_x$;
 \item[(vi)] $\|\xi'_i(x)[u]\|_X\leq \delta'\|u\|_X$, $|J'(x)[u]|\leq \delta'\|u\|_X$ and
 $|J''(x)[u,w]| \leq \delta' \|u\|_X\|w\|_X$ for every  $u,w\in X$.
\end{enumerate}
Then there exists $x_0 \in \Neh$ such that $J(x_0)=c$ and $J'(x_0)=0$.
\end{teo}
\begin{proof}
First of all, by virtue of Remark \ref{rem:only_min_sequences}, we can work
in the sublevel of $J$. We choose
\[
\Vb^+ := V^+, \qquad \Vb^-:=\R^h
\]
(which have trivial intersection by (v)), and $\tau_x^+$ to be the identity, $\tau_x^-:V_x^-\to\R^h$ defined as
\[
\tau_x^- : \xi \mapsto \left(\frac{\langle\xi ,\xi_1(x)\rangle_X }{\|\xi_1(x)\|_X},\ldots,
\frac{\langle\xi ,\xi_h(x)\rangle_X }{\|\xi_h(x)\|_X}\right).
\]
In particular, assumption (iii) immediately implies that $\tau_x^-$ is an isometric
isomorphism. Taking into account Theorem \ref{teo:main_intro} and Corollary \ref{coro:abst_neh},
the only non-trivial things to check are that assumption (inv) holds and that $\partial_x(\tau^-_x)^{-1}$ is
uniformly bounded as a bilinear map on $\Vb^-\times X$. On one hand, if
$\xi(x)=\sum_i t_i(x) \xi_i(x)$, then for any $v\in V_x$ it holds
\[
\xi'(x)[v] = \sum_{i=1}^h
t'_i(x)[v]\,\xi_i(x) + \sum_{i=1}^h t_i(x) \xi_i'(x)[v],
\]
where the first term belongs to $V_x^-$, while the second one is an element of $V_x$ by assumption
(iv). On the other hand, if $t\in \R^h$ and $u\in X$, then
\[
\partial_x(\tau^-_x)^{-1}: (t,u)\mapsto \sum_{i=1}^h t_i\left(\frac{\xi'_i(x)[u]}{\|\xi_i(x)\|}
-\frac{\langle \xi_i(x),\xi'_i(x)[u]\rangle \xi_i(x)}{\|\xi_i(x)\|^3}
\right),
\]
which is uniformly bounded by assumptions (iii) and (vi).
\end{proof}

\section{Superlinear elliptic systems}

In this section we apply Theorem \ref{teo:applications_specialized} in order to obtain multiple
positive solutions for the system
\begin{equation}\label{eq:system}
-\Delta u_i=\partial_iF(u_1,\ldots,u_k), \qquad i=1,\ldots,k,
\end{equation}
where every $u_i$ is $H^1_0$ on a bounded regular domain $\Omega\subset\R^N$.
We stress that, with ``positive solutions'', we mean that
\emph{every} component $u_i$ must be non negative and non identically zero.
We denote by $e_1,\dots,e_k$ the canonical base of $\R^k$, so that
\[
u=(u_1,\dots,u_k)=\sum_i u_ie_i.
\]
Throughout this section we will assume that $F\in C^2(\R^k,\R)$ and that there exist $p\in(2,2^*)$,
$C_F>0$  and $\delta>0$ such that, for every $u,\lambda \in\R^k$, it holds
\begin{itemize}
\item[(F1)] $\sum_{i,j} |\partial^2_{ij}F(u)|\leq C_F |u|^{p-2}$,
$\sum_{i}|\partial_iF(u)|\leq C_F|u|^{p-1}$ and $|F(u)|\leq C_F|u|^{p}$;
\item[(F2)] $\sum_{i,j}\partial^2_{ij}F(u)\lambda_i u_i\lambda_j u_j-
(1+\delta)\sum_i\partial_iF(u)\lambda_i^2u_i\geq 0$;
\item[(F3)] $\partial_iF(u)u_i\leq \partial_iF(u_i e_i)u_i$ for every $i$;
\item[(F4)] for every $i$ there exists $\bar u_i>0$ such that $\partial_i F(\bar u_i e_i) > 0$ .
\end{itemize}
Assumptions (F1),(F2)  and (F4) are quite standard when searching for solutions
of elliptic problems with variational methods, even in the case of one single equation.
As it concerns (F3), it can be slightly weakened (see the proof of Theorem \ref{teo:intro_appl}
at the end of the paper), and completely neglected
in case one admits solutions with some vanishing components (see Remark \ref{rem:vanishing_comp}).
Under these assumptions one can easily obtain some further inequalities, such as the
classical Ambrosetti-Rabinowitz condition
\begin{equation}\label{eq:ambrosetti_rabinowitz}
\nabla F(u)\cdot u-(2+\delta)F(u)\geq 0
\end{equation}
(notice that, by (F2), the function $t\mapsto \nabla F(tu)\cdot tu -(2+\delta)F(tu)$ is nondecreasing for
$t\in (0,1)$) and
\begin{equation}\label{eq:rapporto_monotono}
\partial_i F(u_ie_i)u_i \geq \frac{\partial_iF(\bar u_i e_i)\bar u_i}{\bar u_i^{2+\delta}}
u_i^{2+\delta}\quad
\text{for }u_i\geq \bar u_i
\end{equation}
(again by (F2), the function $t\mapsto \partial_iF(te_i)t/t^{2+\delta}$ is nondecreasing for $t> 0$).
Notice that the solutions of \eqref{eq:system} can be seen as critical points of the energy functional
\[
J(u)=\frac{1}{2}\int_{\Omega}|\nabla u|^2\,dx -\int_{\Omega}F(u)\,dx.
\]
It is standard to prove that $J\in C^2(X,\R)$ where $X:=H^1_0(\Omega,\R^k)$
is endowed with the norm $\|u\|^2=\int_{\Omega}|\nabla u|^2\,dx=\sum_i\int_{\Omega}
|\nabla u_i|^2\,dx$. Since we search for
positive solutions, we assume without loss of generality that $F$ is even with respect
to each component. All solutions will be found as minimizers of $J$ on suitable
even constraints. By standard arguments we obtain that, for any minimizer $(u_1,\dots,u_k)$
with $u_i\neq0$, then also $(|u_1|,\dots,|u_k|)$ is a minimizer, which components are
strictly positive by the strong maximum principle. For this reason, with a slight abuse,
from now on we will work only with $k$-tuples having non-negative components.

\subsection{Ground states}

We start investigating the existence of ground state solutions.
Such a problem has already been successfully faced in \cite{ctv2003JFA,dancerWeiWeth2010}, nonetheless we prefer to prove the result as a direct application of Theorem \ref{teo:applications_specialized}. This will be useful in the following, where we turn to the
analysis of excited states.
\begin{teo}\label{teo:application_ground_states}
Let $F\in C^2(\R^k,\R)$ satisfy (F1)-(F4). Then there exists a positive solution of \eqref{eq:system}
in $H^1_0(\Omega)$.
\end{teo}
Before proving this result, we state in the following lemma some preliminary
estimates which will be useful also in the next subsections.
\begin{lemma}\label{lem:livello>norma}
Let $u\in X$ be such that $J'(u)[u_ie_i]=0$ for every $i=1,\dots,k$. Then
\[
J(u)\geq \frac{\delta}{4+2\delta}\|u\|^2.
\]
Moreover, denoting by $C_S(\Omega,p)$ the Sobolev constant of the
embedding $H^1_0(\Omega)\subset L^p(\Omega)$,
\[
\text{either }\|u_i\|\geq (C_F C_S(\Omega,p)^p)^{-1/(p-2)}\text{ or }u_i\equiv0.
\]
\end{lemma}
\begin{proof}
Recalling the definition of $J$, the assumption writes
\[
\int_{\Omega} |\nabla u_i|^2\,dx=\int_{\Omega}\partial_i F(u)u_i\,dx,
\]
for every $i$. As it regards the first part, using equation
\eqref{eq:ambrosetti_rabinowitz} we have that
\begin{equation*}
J(u)\geq \frac{1}{2}\int_{\Omega}|\nabla u|^2\,dx-\frac{1}{2+\delta}\int_{\Omega} \nabla F(u)\cdot u \,dx=\frac{\delta}{4+2\delta}\|u\|^2.
\end{equation*}
On the other hand, assumptions (F3) and (F1) give
\[
\begin{split}
\int_{\Omega}|\nabla u_i|^2\,dx&=\int_{\Omega} \partial_i F(u)u_i\,dx \leq
\int_{\Omega} \partial_i F(u_ie_i)u_i \,dx \\
&\leq C_F\int_{\Omega}|u_i|^p\,dx \leq
C_F C_S(\Omega,p)^p\left(\int_{\Omega}|\nabla u_i|^2\,dx\right)^{p/2}.\qedhere
\end{split}
\]
\end{proof}
\begin{proof}[Proof of Theorem \ref{teo:application_ground_states}]
We define
\[
A=\{ u\in X:\ u_i\not\equiv 0 \ \text{ for every } i
\}
\]
and
\[
V^+=\left\{0\right\},\qquad\xi_i(u)=u_ie_i,\, i=1,\ldots,k.
\]
Within this setting we have
\begin{equation*}
\Neh=\left\{ u\in A:\ \int_{\Omega} |\nabla u_i|^2\,dx=\int_{\Omega}\partial_i F(u)u_i\,dx,
\quad i=1,\ldots,k \right\},
\end{equation*}
so that Lemma \ref{lem:livello>norma} holds true for any of its elements. Let us check the assumptions of Theorem \ref{teo:applications_specialized}.

(i) The first part of Lemma \ref{lem:livello>norma} shows that $c\geq 0$, while the second
part implies that $\overline{\Neh}\setminus\Neh = \emptyset$, thus the only thing to prove is that
$c<+\infty$, that is $\Neh\neq\emptyset$. To this aim let $u\in X$ be fixed in such a way
that $u_i\geq0$,  $u_i\not \equiv 0$, $u_i\cdot u_j\equiv0$ for $i\neq j$. We claim that there
exists $\lambda\in\R^k$, with all positive components, such that
$(\lambda_1 u_1,\ldots,\lambda_k u_k) \in\Neh$.
For each $i$ let us define the smooth function
\[
g_i(\lambda_i) := \lambda_i^2 \int_{\Omega}|\nabla u_i|^2\,dx- \int_{\Omega} \partial_i F(\lambda_i u_i e_i)\lambda_i u_i \,dx,
\]
so that the claim is equivalent to the existence of $\lambda$ such that $g_i(\lambda_i)=0$ for every
$i$.
On one hand, by assumption (F1) we have
\[
g_i(\lambda_i) \geq \lambda_i^2\int_{\Omega}|\nabla u_i|^2\,dx- \lambda_i^p C_F \int_{\Omega} u_i^p
\,dx,
\]
which is positive for $\lambda_i$ small. On the other hand, \eqref{eq:rapporto_monotono} implies
\[
\begin{split}
g_i(\lambda_i) &\leq
\lambda_i^2\int_{\Omega}|\nabla u_i|^2\,dx + C -
\int_{\{\lambda_iu_i\geq\bar u_i\}} \partial_i F(\lambda_i u_i e_i)\lambda_i u_i \,dx \\
& \leq \lambda_i^2\int_{\Omega}|\nabla u_i|^2\,dx + C -
\int_{\{\lambda_iu_i\geq\bar u_i\}} \frac{\partial_i F(\bar u_i e_i)\bar u_i}{\bar u_i^{2+\delta}}
(\lambda_i u_i)^{2+\delta} \,dx,
\end{split}
\]
which, by (F4), is negative for $\lambda_i$ sufficiently large.

(ii) It is a standard consequence of equation \eqref{eq:ambrosetti_rabinowitz} (see for example \cite{struwe08book}).

(iii) On one hand it is trivial to check that the $\xi_i$'s are orthogonal, on the other hand
Lemma \ref{lem:livello>norma} implies that $\|\xi_i(u)\|\geq \delta >0$.

(iv) Given $u\in\Neh$, any vector belonging to $V_u$ has the form $v=(\lambda_1 u_1,\ldots,\lambda_k u_k)$ for some $\lambda\in\R^k$.
Hence $\xi_i'(u)[v]=\lambda_i u_i e_i \ \in V_u$.

(v)  Let $u\in\Neh$ and $v=(\lambda_1 u_1,\ldots,\lambda_k u_k)\in V_u$. Assumption (F2) and the
definition of $\Neh$ provide
\[
J''(u)[v,v]\leq \sum_i \int_{\Omega}\lambda_i^2 |\nabla u_i|^2\,dx -(1+\delta)\sum_i \int_{\Omega} \partial_i F(u)\lambda_i^2 u_i\,dx
=-\delta \|v\|^2.
\]

(vi) Using assumption (F1), H\"older inequality and Sobolev embedding we have,
for every $v,w\in X$, and $u\in\Neh$,
\[
\begin{split}
\|\xi'_i(u)[v]\| &= \|v_i\| \leq \|v\|,\\
|J'(u)[v]| &\leq \int_{\Omega}|\nabla u||\nabla v|\,dx +C_F\int_{\Omega}|u|^{p-1}|v|\,dx \leq
\left(\|u\|+C\|u\|^{p-1}\right)\|v\|,\\
|J''(u)[v,w]| &\leq \int_{\Omega}|\nabla v||\nabla w|\,dx +C_F\int_{\Omega}|u|^{p-2}|v||w|\,dx \leq
\left(1+C\|u\|^{p-2}\right)\|v\|\|w\|.
\end{split}
\]
We can easily conclude observing that, by Lemma \ref{lem:livello>norma}, $\|u\|$ is uniformly bounded on
$\Neh\cap\{J\leq c+1\}$.
\end{proof}
\begin{coro}\label{coro:gs}
For every $I\subset\left\{1,\dots,k\right\}$, $\tilde\Omega\subset\R^N$ smooth and bounded domain,
there exists a (minimal energy) solution of \eqref{eq:system} in $H^1_0(\tilde\Omega)$ such that
$u_i>0$ for $i\in I$, $u_i\equiv0$ otherwise.
\end{coro}
\begin{proof}
It suffices to observe that, letting $\tilde k = \# I$ and $\sigma:\left\{1,\dots,\tilde k\right\}\to I$
increasing, then
\[
\tilde F (\tilde u_1, \dots, \tilde u_{\tilde k})= F\left(\sum_{i\in I} \tilde u_i e_{\sigma(i)}\right)
\]
satisfies (F1)-(F4) on $\R^{\tilde k}$.
\end{proof}
\begin{rem}\label{rem:vanishing_comp}
Neglecting assumption (F3), it is possible to use the standard Nehari manifold
in order find nontrivial solutions with possibly vanishing components. In the setting above,
this corresponds to replacing $\spann\{u_1e_1,\dots,u_ke_k\}$ with $\spann\{u\}$. Indeed,
assumption (F3) is used only in the second part of Lemma \ref{lem:livello>norma}, which argument
can be directly applied to $\int_\Omega|\nabla u|^2\,dx$. The same idea can be carried on
also in the results below.
\end{rem}

\subsection{Multi-bump solutions}

We will prove multiplicity of positive solutions for system \eqref{eq:system} when $\Omega$
is close to the union of disjoint subdomains. More precisely we introduce the following
notations and assumptions.
\begin{itemize}
 \item[($\Omega$1)] $\Omega$ and $\Omega_l$, $l=1,\dots,n$, are bounded regular domains and $D$
 is a bounded regular open set, such that
\[
\overline{\Omega}_l\cap \overline{\Omega}_m =\emptyset\text{ for every }
l\neq m,\qquad \Omega \setminus \overline{D} = \bigcup_{l=1}^n \Omega_l,
\]
 \item[($\Omega$2)] $B\supset \Omega$ is a fixed ball, $\Gamma_l \subsetneq \partial \Omega_l$, $l=1,\dots,n$,
 are (non-empty and) relatively open, such that
 \[
 \partial D \cap\Gamma_l=\emptyset
 \]
 \item[($\Omega$3)] $\eta_l\in C^\infty(\R^N)$, $l=1,\dots,n$, are such that $0\leq \eta_l \leq 1$,
$\eta_l|_{\Omega_l}=1$, and $\eta_l\cdot\eta_m\equiv 0$ for $l\neq m$. $C_\eta>0$ denotes a constant
(depending only on $D$, $\eta_1,\dots,\eta_n$) with the property that
\begin{equation}
\int_{\Omega}|\nabla \eta_l|^2  \vphi^2\,dx \leq C_\eta \int_{\Omega}|\nabla \vphi|^2\,dx
\quad\text{for every }\vphi\in H^1_0(\Omega)
\end{equation}
(observe that the first integral is actually on $D$).
\end{itemize}
In our construction, we assume $\Omega_l$, $\Gamma_l$ and $B$ to be fixed, while $D$,
and hence $\Omega$ and $\eta_l$, to vary. From this point of view, since $H^1_0(\Omega)
\subset H^1_0(B)$, the role of $B$ is only to provide Sobolev constants not depending
on $D$, neither on $\Omega$.
As we mentioned, we consider the case in which
$D$ is suitably small, meaning that both the Lebesgue measure $|D|$ and the constant $C_\eta$
above are small. This last property is related to the smallness
of the $N$-capacity of suitable subsets of $D$, and it can be shown to hold, for instance,
if $D$ can be decomposed in a finite number of parts, each of which
lies between two hyperplanes sufficiently close.

We are going to distinguish different solutions of \eqref{eq:system} by prescribing the ``size''
of $u_i|_{\Omega_l}$, for every $i=1,\dots,k$ and $l=1,\dots,n$. More precisely let us fix any
\begin{equation}\label{eq:L_i}
L_i\subset\{1,\dots,n\},\,L_i\neq\emptyset,\qquad i=1,\dots,k.
\end{equation}
We will provide a solution such that $u_i|_{\Omega_l}$ is ``large'' for $l\in L_i$ and ``small''
for $l\not\in L_i$. Due to the arbitrary choice of the sets $L_i$'s, this will imply
the existence of $(2^n-1)^k$ different positive solutions of system \eqref{eq:system}.
The size of each bump will be classified in relation to the constants
\[
r_l:=\left(\frac{p}{2}C_F C_S(\Omega_l,\Gamma_l,p)^p\right)^{-1/(p-2)},
\]
where $C_S(\Omega_l,\Gamma_l,p)$ is the Sobolev constant of the embedding
$H^1_{0,\Gamma}(\Omega_l)\subset L^p(\Omega_l)$ (compare with the constant which appears in
Lemma \ref{lem:livello>norma}). Let us remark that $r_l$ is independent of $D$. We can finally
state the main result of this section.
\begin{teo}\label{teo:application_single_bump}
Let $F\in C^2(\R^k,\R)$ satisfy (F1)-(F4) and let $\Omega\subset\R^N$ satisfy ($\Omega$1)-($\Omega$3).
Assume that the quantities
\[
|D|,\,C_\eta
\text{ are sufficiently small.}
\]
Then for any $L_1,\dots,L_k$ as in \eqref{eq:L_i} there exists a positive solution $u$ of
\eqref{eq:system} such that, for every $i$ and $l$,
\[
\int_{\Omega_l}|\nabla u_i|^2\,dx> r_l^2\text{ for }l\in L_i, \qquad
\int_{\Omega_l}|\nabla u_i|^2\,dx< r_l^2\text{ for }l\not\in L_i.
\]
\end{teo}
To start with, using the results of the previous subsection, it is easy to provide a $k$-tuple $g$
of non-negative functions in $H^1_0(\cup_l \Omega_l)$ such that
\begin{equation*}
-\Delta g_i=\partial_i F(g_1,\ldots,g_k), \qquad \text{and } g_i|_{\Omega_l}\text{ is either positive or zero,}
\end{equation*}
depending on whether $l \in L_i$ or not (in some sense, one can think of $g$ as the required solution, in the singular limit case $D=\emptyset$). Indeed, for any $l$ one can apply Corollary \ref{coro:gs} with $\tilde\Omega = \Omega_l$ and $I=\{i: l\in L_i\}$. Then $g$ is the sum of the corresponding solutions. By trivial extension, $g\in H^1_0(\Omega)$.

Let us define the constant (independent of $D$)
\[
R^2:=\max\left\{\|g\|^2, \frac{4+2\delta}{\delta} J(g)\right\}+1,
\]
where $\delta$ has been introduced in assumption (F2).
In order to apply Theorem \ref{teo:applications_specialized}
we define
\[
V^+ := \left\{v\in H^1_0(\Omega,\R^k):\,v_i\in H^1_0\left(\Omega\setminus\bigcup_{l\in L_i}\overline{\Omega}_l)\right)\right\}
\]
and
\[
\xi_{i,l}(u):=\eta_l u_ie_i, \quad i=1,\dots,k\text{ and }l\in L_i,
\]
the latter being smooth on
\[
A:=\left\{ u\in X:\, \|u\|<R,\begin{array}{l} \int_{\Omega_l}|\nabla u_i|^2\,dx>r_l^2\text{ if }l\in L_i,
\smallskip\\
\int_{\Omega_l}|\nabla u_i|^2\,dx<r_l^2\text{ if }l\not\in L_i
\end{array}
\right\}.
\]
On one hand we have that
\[
u_ie_i=\left(1-\sum_{l\in L_i}\eta_l\right)u_ie_i+\sum_{l\in L_i}\eta_l u_ie_i \in V_u,
\]
since the first term is in $V^+$ and the second one in $V^-_u$. This in particular
implies, for every $i$,
\begin{equation}\label{eq:nehari_vecchia}
u\in\Neh\quad\implies\quad J'(u)[u_ie_i]=0.
\end{equation}
Analogously, for every $i$ and $l$,
\begin{equation}\label{eq:inv}
\eta_l^2u_ie_i \in V_u,
\end{equation}
indeed, either it belongs to $V^+$ if $l\not\in L_i$, or it is equal to
$(\eta_l^2-\eta_l)u_ie_i+\eta_l u_ie_i$, the former belonging to $V^+$
and the latter to $V^-_u$. We
deduce that, for every $u\in\Neh$ and $l=1,\dots,n$, it holds
\[
0=J'(u)[\eta_l^2u_ie_i]=\int_\Omega \nabla u_i \cdot \nabla(\eta_l^2 u_i)\,dx- \int_\Omega
\partial_i F(u)\eta_l^2u_i\,dx,
\]
which implies
\begin{equation}\label{eq:nehari_perturbata}
\int_{\Omega}|\nabla (\eta_l u_i)|^2\,dx =
\int_{\Omega} \partial_iF(u)\eta_l^2 u_i\,dx + \int_{\Omega}|\nabla \eta_l|^2  u_i^2\,dx.
\end{equation}
Using this property we can prove a result which can be seen as a perturbation
of the second part of Lemma \ref{lem:livello>norma}. Such result will allow
to better localize the bumps of the elements of $\Neh$.
\begin{lemma}\label{lem:lontano_da_zero_perturbato}
Let $|D|$, $C_\eta$ be sufficiently small. Then there exist positive
constants $C$, $\eps$ such that, for every $u\in\Neh$, it holds
\[
\begin{split}
 &\int_{\Omega_l}|\nabla u_i|^2\,dx > r_l^2 \quad\implies\quad \int_{\Omega_l}|\nabla u_i|^2\,dx
 \geq (1+C) r_l^2\\
 &\int_{\Omega_l}|\nabla u_i|^2\,dx < r_l^2 \quad\implies\quad \int_{\Omega_l}|\nabla u_i|^2\,dx
 \leq \eps^2,
\end{split}
\]
where $\eps$ can be made arbitrarily small with $|D|$, $C_\eta$.
\end{lemma}
\begin{proof}
Using \eqref{eq:nehari_perturbata}, (F3) and (F1) we have
\[
\begin{split}
\int_{\Omega_l}|\nabla u_i|^2\,dx &\leq \int_{\Omega}|\nabla (\eta_l u_i)|^2\,dx =
\int_{\Omega} \partial_i F(u)\eta_l^2 u_i\,dx + \int_{\Omega}|\nabla \eta_l|^2  u_i^2\,dx\\
&= \int_{\Omega_l} \partial_i F(u)u_i\,dx +\int_{D} \partial_i F(u)\eta_l^2 u_i\,dx +\int_{\Omega}|\nabla \eta_l|^2  u_i^2\,dx\\
&\leq \int_{\Omega_l} \partial_i F(u_i e_i)u_i\,dx +C_F \int_{D} |u|^p\,dx + C_\eta R^2\\
&\leq C_F\int_{\Omega_l} |u_i|^p\,dx +C_F |D|^{(2^*-p)/2^*} C_S(B,2^*)^pR^{p} + C_\eta R^2\\
&\leq C_F C_S(\Omega_l,\Gamma_l,p)^p\left(\int_{\Omega_l} |\nabla u_i|^2\,dx \right)^{p/2} +
\eps',
\end{split}
\]
where $\eps'$ denotes a quantity arbitrarily small whenever $|D|$ and $C_\eta$ are.
The conclusion easily follows by observing that, denoting by
\[
h(t)=C_F C_S(\Omega_l,\Gamma_l,p)^p t^p - t^2 + \eps',
\]
it holds $h'(r_l)=0$ and $h(r_l)<0$ for $\eps'$ small.
\end{proof}
\begin{proof}[End of the proof of Theorem \ref{teo:application_single_bump}]
We check the assumptions of Theorem \ref{teo:applications_specialized}.

(i) To start with, we have that $c<+\infty$, since $g\in\Neh$. Secondly,
by equation \eqref{eq:nehari_vecchia}, we have that Lemma \ref{lem:livello>norma}
holds true also in the present case, thus providing $c\geq0$. Finally, let
$u\in \overline{\Neh}\setminus \Neh$: then, by Lemma \ref{lem:lontano_da_zero_perturbato}
necessarily $\|u\|=R$. But then, using again Lemma \ref{lem:livello>norma} and the definition of $R$ we obtain
\[
J(u)\geq \frac{\delta}{4+2\delta} R^2 > J( g)\geq c.
\]

(ii) The same as in the previous subsection.

(iii) By definition of $A$ we have that $\|\xi_{i,l}(u)\|> r_l$ for every $i$, $l\in L_i$.

(iv) It follows from \eqref{eq:inv}.

(v) If $u\in\Neh$ and $v\in V^+$ then $v_i\equiv 0$ on $\Omega_l$ for every $l\in L_i$, whereas for $l\not\in L_i$ it holds
$\int_{\Omega_l} |\nabla v_i|^2\,dx<\epsilon^2$ where $\eps$ is defined as in Lemma \ref{lem:lontano_da_zero_perturbato}. Hence we have
\begin{multline*}
J''(u)[v,v] =
  \int_{\Omega} |\nabla v|^2\,dx - \int_{D} \sum_{i,j} \partial^2_{ij} F(u)v_iv_j\,dx- \sum_{l\not\in L_i}\int_{\Omega_l}\sum_{i,j} \partial^2_{ij} F(u)v_iv_j\,dx\\
  \geq \left(1 - C_F |D|^{(2^*-p)/2^*} C_S(B,2^*)^pR^{p-2}
  - \sum_{l\not\in L_i} C_F C_S(\Omega_l,\Gamma_l,p)^p \eps^{p-2}\right) \|v\|^2.
\end{multline*}
On the other hand, if $v\in V^-_u$ then $v=\sum_i \left(\sum_{l\in L_i}t_{i,l} \eta_l \right) u_i$, for some $t_{i,l}\in\R$.
Using (F2) and \eqref{eq:nehari_perturbata} we obtain
\[
\begin{split}
J''(u)[v,v] &\leq
  \|v\|^2 - (1+\delta) \int_{\Omega}\sum_i \partial_i F(u) \sum_{l\in L_i} t_{i,l}^2 \eta^2_l u_i \,dx\\
  &= -\delta \|v\|^2 + (1+\delta)\int_\Omega \sum_i \sum_{l\in L_i} t_{i,l}^2 |\nabla\eta_l|^2 u_i^2\,dx \\
  &\leq -\delta \|v\|^2 + (1+\delta)C_\eta \int_\Omega \sum_i \sum_{l\in L_i} t_{i,l}^2 |\nabla u_i|^2\,dx\\
  &=  -\delta \|v\|^2 + (1+\delta)C_\eta \sum_i \sum_{l\in L_i} \frac{\int_\Omega |\nabla u_i|^2\,dx}{\int_{\Omega_l} |\nabla u_i|^2\,dx}
\int_{\Omega_l} t_{i,l}^2|\nabla u_i|^2\,dx\\
  &\leq  -\delta \|v\|^2 + (1+\delta)C_\eta \sum_i \sum_{l\in L_i} \frac{R^2}{r_l^2}
  \int_{\Omega_l} t_{i,l}^2|\nabla u_i|^2\,dx\\
  &\leq \left( -\delta + (1+\delta)C_\eta \frac{R^2}{\min_{l\in L_i}{r_l^2}}\right)\|v\|^2.
\end{split}
\]
In both cases assumption (v) holds true when $|D|$ and $C_\eta$ are sufficiently small.

(vi) the same as in the previous subsection, once one notices that
\[
\|\xi_{i,l}'(u)[v]\|^2 
\leq
2\int_{\Omega} \left(v_i^2|\nabla \eta_l|^2+\eta_l^2|\nabla v_i|^2\right)\,dx \leq (C_\eta+1)
\|v\|^2. \qedhere
\]
\end{proof}
\begin{proof}[Proof of Theorem \ref{teo:intro_appl}]
Since $\beta_{ij}=\beta_{ji}$, system \eqref{eq:bec} is variational, with potential
\[
F(u)=\sum_{i=1}^k\left(\frac{\mu_i}{4} u_i^4 + \sum_{j\neq i}\frac{\beta_{ij}}{4}u_i^2u_j^2\right).
\]
It is easy to check that it
satisfies assumptions (F1), (F2), (F4) with $p=4<2^*$ and $\delta=2$. If $\beta_{ij}\leq 0$
for every $i,j$, then it also satisfies (F3), so that Theorem \ref{teo:application_single_bump}
immediately applies. Since (F3) is used only in the estimate in Lemma \ref{lem:lontano_da_zero_perturbato} (and in its counterpart in Lemma \ref{teo:application_ground_states}) we show how to replace that argument in case
$\beta_{ij}\leq \bar \beta$ for every $i,j$, with $\bar \beta$ positive and sufficiently small.
We have
\[
\int_{\Omega_l} \partial_i F(u)u_i\,dx = \int_{\Omega_l} \left(\mu_i u_i^4 + \sum_{j\neq i} \beta_{ij} u_i^2 u_j^2\right)\,dx \leq \mu_i
\int_{\Omega_l} u_i^4\,dx + \bar\beta C_S^4(B,4)R^4,
\]
where the last term is arbitrarily small when $\bar\beta$ is.
\end{proof}

\noindent\verb"benedetta.noris1@unimib.it"\\
Dipartimento di Matematica e Applicazioni, Universit\`a degli Studi
di Milano-Bicocca, via Bicocca degli Arcimboldi 8, 20126 Milano,
Italy

\noindent \verb"gianmaria.verzini@polimi.it"\\
Dipartimento di Matematica, Politecnico di Milano, p.za Leonardo da
Vinci 32,  20133 Milano, Italy


\begin{thebibliography}{10}

\bibitem{ambrosetti}
A.~Ambrosetti.
\newblock Critical points and nonlinear variational problems.
\newblock {\em M\'em. Soc. Math. France (N.S.)}, \penalty0 (49):\penalty0 139,
  1992.

\bibitem{berger71}
M.~S. Berger.
\newblock Periodic solutions of second order dynamical systems and
  isoperimetric variational problems.
\newblock {\em Amer. J. Math.}, 93:\penalty0 1--10, 1971.

\bibitem{berger77}
M.~S. Berger and M.~Schechter.
\newblock On the solvability of semilinear gradient operator equations.
\newblock {\em Advances in Math.}, 25\penalty0 (2):\penalty0 97--132, 1977.

\bibitem{birkhoff1}
G.~D. Birkhoff and M.~R. Hestenes.
\newblock Natural isoperimetric conditions in the calculus of variations.
\newblock {\em Duke Math. J.}, 1\penalty0 (2):\penalty0 198--286, 1935.

\bibitem{byeon2001}
J.~Byeon.
\newblock Nonlinear elliptic problems on singularly perturbed domains.
\newblock {\em Proc. Roy. Soc. Edinburgh Sect. A}, 131\penalty0 (5):\penalty0
  1023--1037, 2001.

\bibitem{contiFelli09}
M.~Conti and V.~Felli.
\newblock Minimal coexistence configurations for multispecies systems.
\newblock {\em Nonlinear Anal.}, 71\penalty0 (7-8):\penalty0 3163--3175, 2009.

\bibitem{ctv2002poincare}
M.~Conti, S.~Terracini, and G.~Verzini.
\newblock Nehari's problem and competing species systems.
\newblock {\em Ann. Inst. H. Poincar\'e Anal. Non Lin\'eaire}, 19\penalty0
  (6):\penalty0 871--888, 2002.

\bibitem{ctv2003JFA}
M.~Conti, S.~Terracini, and G.~Verzini.
\newblock An optimal partition problem related to nonlinear eigenvalues.
\newblock {\em J. Funct. Anal.}, 198\penalty0 (1):\penalty0 160--196, 2003.

\bibitem{dancer88}
E.~N. Dancer.
\newblock The effect of domain shape on the number of positive solutions of
  certain nonlinear equations.
\newblock {\em J. Differential Equations}, 74\penalty0 (1):\penalty0 120--156,
  1988.

\bibitem{dancer90}
E.~N. Dancer.
\newblock The effect of domain shape on the number of positive solutions of
  certain nonlinear equations. {II}.
\newblock {\em J. Differential Equations}, 87\penalty0 (2):\penalty0 316--339,
  1990.

\bibitem{dancerWeiWeth2010}
E.~N. Dancer, J.~Wei, and T.~Weth.
\newblock A priori bounds versus multiple existence of positive solutions for a
  nonlinear {S}chr\"odinger system.
\newblock {\em Ann. Inst. H. Poincar\'e Anal. Non Lin\'eaire}, 27\penalty0
  (3):\penalty0 953--969, 2010.

\bibitem{giraoGomes09}
P.~M. Gir{\~a}o and J.~M. Gomes.
\newblock Multibump nodal solutions for an indefinite superlinear elliptic
  problem.
\newblock {\em J. Differential Equations}, 247\penalty0 (4):\penalty0
  1001--1012, 2009.

\bibitem{nehari60}
Z.~Nehari.
\newblock On a class of nonlinear second-order differential equations.
\newblock {\em Trans. Amer. Math. Soc.}, 95:\penalty0 101--123, 1960.

\bibitem{nehari61}
Z.~Nehari.
\newblock Characteristic values associated with a class of non-linear
  second-order differential equations.
\newblock {\em Acta Math.}, 105:\penalty0 141--175, 1961.

\bibitem{nehari75}
Z.~Nehari.
\newblock A nonlinear oscillation theorem.
\newblock {\em Duke Math. J.}, 42:\penalty0 183--189, 1975.

\bibitem{palais79}
R.~S. Palais.
\newblock The principle of symmetric criticality.
\newblock {\em Comm. Math. Phys.}, 69\penalty0 (1):\penalty0 19--30, 1979.

\bibitem{pankov07}
A.~Pankov.
\newblock Gap solitons in periodic discrete nonlinear {S}chr\"odinger
  equations. {II}. {A} generalized {N}ehari manifold approach.
\newblock {\em Discrete Contin. Dyn. Syst.}, 19\penalty0 (2):\penalty0
  419--430, 2007.

\bibitem{ramosPistoia04}
A.~Pistoia and M.~Ramos.
\newblock Locating the peaks of the least energy solutions to an elliptic
  system with {N}eumann boundary conditions.
\newblock {\em J. Differential Equations}, 201\penalty0 (1):\penalty0 160--176,
  2004.

\bibitem{poincare05}
H.~Poincar{\'e}.
\newblock Sur les lignes g\'eod\'esiques des surfaces convexes.
\newblock {\em Trans. Amer. Math. Soc.}, 6\penalty0 (3):\penalty0 237--274,
  1905.

\bibitem{ramosTavares2008}
M.~Ramos and H.~Tavares.
\newblock Solutions with multiple spike patterns for an elliptic system.
\newblock {\em Calc. Var. Partial Differential Equations}, 31\penalty0
  (1):\penalty0 1--25, 2008.

\bibitem{ramosYang05}
M.~Ramos and J.~Yang.
\newblock Spike-layered solutions for an elliptic system with {N}eumann
  boundary conditions.
\newblock {\em Trans. Amer. Math. Soc.}, 357\penalty0 (8):\penalty0 3265--3284
  (electronic), 2005.

\bibitem{struwe08book}
M.~Struwe.
\newblock {\em Variational methods}, volume~34 of {\em Ergebnisse der
  Mathematik und ihrer Grenzgebiete. 3. Folge. A Series of Modern Surveys in
  Mathematics [Results in Mathematics and Related Areas. 3rd Series. A Series
  of Modern Surveys in Mathematics]}.
\newblock Springer-Verlag, Berlin, fourth edition, 2008.
\newblock ISBN 978-3-540-74012-4.
\newblock xx+302 pp.
\newblock Applications to nonlinear partial differential equations and
  Hamiltonian systems.

\bibitem{szulkinWeth2009}
A.~Szulkin and T.~Weth.
\newblock Ground state solutions for some indefinite variational problems.
\newblock {\em J. Funct. Anal.}, 257\penalty0 (12):\penalty0 3802--3822, 2009.

\bibitem{szulkinWeth2010}
A.~Szulkin and T.~Weth.
\newblock The method of nehari manifold.
\newblock In {\em Handbook of Nonconvex Analysis and Applications}.
  International Press of Boston, 2010.

\bibitem{vangroesen88}
E.~W.~C. van Groesen.
\newblock Analytical mini-max methods for {H}amiltonian brake orbits of
  prescribed energy.
\newblock {\em J. Math. Anal. Appl.}, 132\penalty0 (1):\penalty0 1--12, 1988.

\end{thebibliography}
\end{document}